\documentclass[11pt]{amsart}
\usepackage{amsfonts,amsmath,amsthm,mathrsfs}
\numberwithin{equation}{section}
\begin{document}
\newtheorem{theorem}{Theorem}[section]
\newtheorem{lemma}{Lemma}[section]
\newtheorem{remark}{Remark}[section]
\newtheorem{proposition}{Proposition}[section]

\newcommand{\bbR}{\mathbb{R}}
\newcommand{\bbC}{\mathbb{C}}
\newcommand{\bbZ}{\mathbb{Z}}

\def\la{\langle}
\def\ra{\rangle}
\def\fac{{\rm !}}

\def\x{\mathbf{x}}
\def\z{\mathbf{x}}
\def\y{\mathbf{y}}
\def\p{\mathbf{p}}
\def\P{\mathbf{P}}
\def\S{\mathbf{S}}
\def\h{\mathbf{h}}
\def\m{\mathbf{m}}
\def\y{\mathbf{y}}
\def\bz{\mathbf{z}}
\def\F{\mathcal{F}}
\def\R{\mathbb{R}}
\def\T{\mathbf{T}}
\def\N{\mathbb{N}}
\def\D{\mathbf{D}}
\def\V{\mathbf{V}}
\def\U{\mathbf{U}}
\def\K{\mathbf{K}}
\def\Q{\mathbf{Q}}
\def\W{\mathbf{W}}
\def\M{\mathbf{M}}
\def\oM{\overline{\mathbf{M}}}

\def\C{\mathbb{C}}
\def\Z{\mathbb{Z}}
\def\bZ{\mathbf{Z}}
\def\H{\mathcal{H}}
\def\A{\mathbf{A}}
\def\V{\mathbf{V}}

\def\B{\mathbf{B}}
\def\c{\mathbf{C}}
\def\L{\mathcal{L}}
\def\bS{\mathbf{S}}
\def\H{\mathcal{H}}
\def\I{\mathbf{I}}
\def\Y{\mathbf{Y}}
\def\X{\mathbf{X}}
\def\G{\mathbf{G}}
\def\f{\mathbf{f}}
\def\z{\mathbf{z}}
\def\bv{\mathbf{v}}
\def\y{\mathbf{y}}
\def\d{\hat{d}}
\def\x{\mathbf{x}}
\def\bI{\mathbf{I}}

\def\g{\mathbf{g}}
\def\w{\mathbf{w}}
\def\b{\mathbf{b}}
\def\a{\mathbf{a}}
\def\u{\mathbf{u}}
\def\v{\mathbf{v}}
\def\q{\mathbf{q}}
\def\e{\mathbf{e}}
\def\s{\mathcal{S}}
\def\cc{\mathcal{C}}
\def\c{\mathbf{c}}
\def\g{\mathbf{g}}

\def\tg{\tilde{g}}
\def\tx{\tilde{\x}}
\def\tg{\tilde{g}}
\def\tA{\tilde{\A}}

\def\cX{\overline{\mathbf{X}}}
\def\bell{\boldsymbol{\ell}}
\def\bxi{\boldsymbol{\xi}}
\def\balpha{\boldsymbol{\alpha}}
\def\bbeta{\boldsymbol{\beta}}
\def\bgamma{\boldsymbol{\gamma}}
\def\eeta{\boldsymbol{\eta}}
\def\bpsi{\boldsymbol{\psi}}
\def\btau{\boldsymbol{\tau}}
\def\supmu{{\rm supp}\,\mu}
\def\supp{{\rm supp}\,}
\def\cd{\mathcal{C}_d}
\def\cok{\mathcal{C}_{\K}}
\def\vol{{\rm vol}\,}
\def\om{\mathbf{\Omega}}

\def\blambda{\boldsymbol{\lambda}}
\def\btheta{\boldsymbol{\theta}}
\def\bphi{\boldsymbol{\phi}}
\def\bpsi{\boldsymbol{\psi}}
\def\bnu{\boldsymbol{\nu}}
\def\bmu{\boldsymbol{\mu}}
\def\bom{\boldsymbol{S}}
\def\tM{\hat{\M}}
\def\tv{\hat{\v}}

\title[Equilibrium vs Lebesgue measure]{Chebyshev and equilibrium measure vs Bernstein and Lebesgue measure}

\author{Jean B. Lasserre}

\address{LAAS-CNRS and Institute of Mathematics, BP 54200, 7 Avenue du Colonel Roche, 31031 Toulouse c\'edex 4, France}

\email{lasserre@laas.fr}

\thanks{J.B. Lasserre is supported by the AI Interdisciplinary Institute ANITI  funding through the french program
``Investing for the Future PI3A" under the grant agreement number ANR-19-PI3A-0004. This research is also part of the programme DesCartes and is supported by the National Research Foundation, Prime Minister's Office, Singapore under its Campus for Research Excellence and Technological Enterprise (CREATE) programme.}

\begin{abstract}
We show that Bernstein polynomials are related to
the Lebesgue measure on $[0,1]$ in a manner similar as Chebyshev polynomials are 
related to the equilibrium measure $dx/\pi\sqrt{1-x^2}$ of $[-1,1]$. 
We also show that Pell's polynomial equation satisfied by  Chebyshev polynomials, provides a partition of unity of $[-1,1]$, the  analogue of 
the partition of unity of $[0,1]$ provided by Bernstein polynomials.
Both partitions of unity
are interpreted as a specific 
algebraic \emph{certificate} that the constant polynomial $``1"$ is positive -- on $[-1,1]$ via Putinar's certificate of positivity (for Chebyshev), and -- on $[0,1]$ via Handeman's certificate of positivity (for Bernstein).
Then in a second step, one combines this partition of unity 
with an interpretation of a duality result of Nesterov in convex conic optimization 
to obtain an explicit connection with the equilibrium measure on $[-1,1]$ (for Chebyshev) and 
Lebesgue measure on $[0,1]$ (for Bernstein). Finally this connection is also partially established for 
the simplex in $\R^d$. 
 \end{abstract}

\maketitle

\section{Introduction}

In a recent contribution \cite{cras-3} we have considered some specific sets $S\subset\R^d$ like the unit box $[-1,1]^d$,
the Euclidean unit ball and the canonical simplex of $\R^d$, and established (in the author's opinion) surprising
connections between 
the Christoffel function of their associated equilibrium measure, the polynomial Pell's equation,
and a Putinar's certificate of positivity on $S$ for the constant polynomial $``1"$. 

The notion of equilibrium measure associated to a given set,
originates from logarithmic potential theory (working in
$\mathbb{C}$ in the univariate case) to minimize some energy functional.
For instance, the equilibrium (Chebsyshev) measure $d\phi:=dx/\pi\sqrt{1-x^2}$ 
minimizes the Riesz $s$-energy functional 
\[\int \int \frac{1}{\vert x-y\vert^s}\,d\mu(x)\,d\mu(y)\,\]
with $s=2$, among all measures $\mu$ equivalent to $\phi$.
Some generalizations have been obtained in the multivariate case 
via pluripotential theory in $\mathbb{C}^n$. 
In particular, if $S\subset\R^n\subset\mathbb{C}^n$ is compact then 
its equilibrium measure is equivalent to
Lebesgue measure on compact subsets of $\mathrm{int}(S)$; see e.g. \cite{bedford}. For the interested reader, some examples of equilibrium measures can be found in e.g.
\cite{dunkl,book}.

For illustration and ease of exposition, consider the
prototypal example of the univariate unit box $[-1,1]$ and its associated equilibrium measure $d\phi=dx/\pi\sqrt{1-x^2}$. 
Starting from the polynomial Pell's equation\footnote{A multivariate polynomial $F\in\Z[\x]$ is called a multi-variable Fermat-Pell polynomial if there exist polynomials $C,H\in\Z[\x]$ such that $C^2-F\,H^2=1$ or $C^2-F\,H^2=-1$ 
for all $\x$. Then the triple $(C,H,F)$ is a multi-variable solution to Pell's equation; see e.g. \cite{pell-2}.}
\begin{equation}
\label{eq:pell}
T_n(x)^2+(1-x^2)\,U_{n-1}(x)^2\,=\,1\,,\quad\forall x\in\R\,,\end{equation}
satisfied by the Chebyshev polynomials $(T_n)_{n\in\N}$ of the first kind 
and $(U_n)_{n\in\N}$ of second kind,  one easily obtains that
\begin{equation}
\label{unity-1}
1\,=\,\sum_{j=0}^nT_j(x)^2/(n+1)+(1-x^2)\,\sum_{i=0}^{n-1}U_{i}(x)^2/(n+1)\,,\quad\forall x\in\R\,.
\end{equation}
Interestingly, \eqref{unity-1} is a sum-of-squares (SOS)-based Putinar's certificate 
\begin{equation}
\label{eq:put}
1\,=\,\sigma_0(x)+(1-x^2)\,\sigma_1(x)\,,\quad \forall x\in\R\,,\end{equation}
that the constant polynomial $``1"$ is positive on $[-1,1]$.
As shown in \cite{cras-1,cras-3}, among all such representations \eqref{eq:put}, the particular form \eqref{unity-1} maximizes an entropy related functional of the Gram matrices of the SOS weights $\sigma_0$ and $\sigma_1$ in \eqref{eq:put}.
On the other hand, \eqref{eq:pell} is nothing less that the Markov-Luk\'acs representation of the constant polynomial 
$``1"$ into a weighted sum of \emph{only two squares}, that is, a representation of the form \eqref{eq:put} 
with single squares instead of sum-of-squares.

Finally, with $x\mapsto g(x):=1-x^2$, and after a rescaling of $T_j$ to $\widehat{T}_j=T_j/\sqrt{2}$ (resp. 
$\widehat{U}_j:=U_j/\sqrt{2}$) so as to obtain a family of polynomials that are \emph{orthonormal} w.r.t. $\phi$
(resp. w.r.t. $g\cdot\phi$ where $g\cdot\phi$ is the measure $gd\phi$), 
\begin{eqnarray}
\label{unity-2}
2n+1&=&\sum_{j=0}^n\widehat{T}_j^2+g\,\sum_{i=0}^{n-1}\widehat{U}_{i}^2\\
\label{unity-3}
&=&\Lambda^{\phi}_n(x)^{-1}+g(x)\,\Lambda^{g\cdot\phi}_n(x)^{-1}\,,\quad\forall x\in\R\,,
\end{eqnarray}
where $\Lambda^\phi_n$ (resp. $\Lambda^{g\cdot\phi}_n$ ) is the degree-$n$ Christoffel function associated with $\phi$ (resp. $g\cdot\phi$); see \cite{cras-3}.
Notice that in \eqref{unity-1}, the polynomials 
\[\left\{(T_j/(n+1))_{j\leq n}\,,\,g\,U_j/(n+1))_{j\leq n-1}\right\}\]
or in \eqref{unity-2}, the polynomials 
\[\left\{(\widehat{T_j}/(2n+1))_{j\leq n}\,,\,g\,\widehat{U}_j/(2n+1))_{j\leq n-1}\right\}\,,\]
form a partition of unity of the interval $[-1,1]$. This partition of unity \eqref{unity-1} (or \eqref{unity-2}) is related \emph{explicitly} to the equilibrium measure $\phi$ of $[-1,1]$ by the interpretation \eqref{unity-3} of \eqref{unity-2} (and/or the orthogonality w.r.t. $\phi$).

\subsection*{Contribution} Inspired by the partition of unity \eqref{unity-1}, we now consider 
another well-known partition of unity, namely:
\begin{equation}
\label{bern-1}
1\,=\,\sum_{j=0}^nB_{n,j}(x)\,,\quad \forall x\in\R\,,\quad\forall n\in\N\,,\end{equation}
of the interval $[0,1]$, provided by the Bernstein polynomials
$(B_{n,j})_{j\leq n}$, where $x\mapsto B_{n,j}(x):={n\choose j}x^j\,(1-x)^{n-j}$, for all $j=0,\ldots,n$.  In particular,
using that $\int_0^1B_{n,j}(x)dx=1/(n+1)$ for all $j=0,\ldots,n$, and summing up, yields 

\begin{eqnarray}
\label{bern-2}
1&=&\frac{2}{(n+1)(n+2)}\,\sum_{t=0}^n\sum_{j=0}^t\frac{B_{t,j}(x)}{\int_0^1B_{t,j}(x)dx}
\,,\quad\forall x\in\R\,,\\
\label{bern-3}
&=&\frac{2}{(n+1)(n+2)}\,\sum_{(i,j)\in\N^2_n}c^*_{ij}x^i \,(1-x)^j\,,\quad\forall x\in\R\,,
\end{eqnarray}
with $1/c^*_{ij}=\frac{i\mathrm{!}\,j\mathrm{!}}{(i+j+1)\mathrm{!}}\,=\,\int_0^1 x^i (1-x)^jdx$.

We want to convince the reader that \eqref{bern-2} (or \eqref{bern-3}) is the analogue for Bernstein polynomials on $S=[0,1]$, of \eqref{unity-2}-\eqref{unity-3} for
Chebyshev polynomials on $S=[-1,1]$. Indeed :

$\bullet$ In \eqref{bern-2}, $(n+1)(n+2)/2$ is the number of terms $x^i (1-x)^j$, 
exactly as $2n+1$ is also the number of terms $\widehat{T}_{j}(x)^2$ and $(1-x^2)\widehat{U}_{j}(x)^2$ in the right-hand-side of \eqref{unity-2}. So in both cases, the polynomial $``1"$ is expressed as an \emph{average} of a certain number of polynomials that are positive on $S$.

$\bullet$ The coefficient $c^*_{ij}$ associated with $x^i(1-x)^j$ (equivalently to $B_{i+j,i}$) is just $1/\int_0^1x^i(1-x)^jdx$
(integration w.r.t. Lebesgue measure on $[0,1]$), 
exactly as $1$ is the coefficient associated with $\widehat{T}_j^2$ and $(1-x^2)\widehat{U}_j^2$, and satisfies
\[1\,=\,\int_{-1}^1\widehat{T}_j(x)^2d\phi\,;\quad
1\,=\,\int_{-1}^1(1-x^2)\widehat{U}_j(x)^2d\phi\,=\,\int_{-1}^1\widehat{U}_j(x)^2\sqrt{1-x^2}dx/\pi\,\]
(integration w.r.t. equilibrium measure $\phi$ on $[-1,1]$).

$\bullet$ The vector of coefficients $\c^*=(c^*_{ij})_{(i,j)\in\N^2_n}$ is the unique optimal solution of the  ``max-entropy" optimization problem:
\[\sup_{\c\geq0}\,\{\,\sum_{(i,j)\in\N^2_n}\log(c_{ij}):\: 1\,=\,\frac{2}{(n+1)(n+2)}\,
\sum_{(i,j)\in\N^2_n}c_{ij}x^i \,(1-x)^j\,,\quad \forall x\in\R\,\}\,.\]
Similarly, with $\v_n(x):=(x^j)_{j\leq n}\in\R[x]^{n+1}$ (and $x\mapsto g(x):=1-x^2$),
\begin{eqnarray*}
x\mapsto \displaystyle\sum_{j=0}^n\widehat{T}_j^2(x)&=&\v_n(x)^T\M_n(\phi)^{-1}\v_n(x)\,,\quad\forall x\,\in\R\\
x\mapsto \displaystyle\sum_{j=0}^{n-1}\widehat{U}_j^2(x)&=&\v_{n-1}(x)^T\M_n(g\cdot\phi)^{-1}\v_{n-1}(x)\,,\quad\forall x\,\in\R,
\end{eqnarray*}
the couple of Gram matrices $(\A^*,\B^*):=(\M_n(\phi)^{-1}, \M_{n-1}(g\cdot\,\phi)^{-1})$
is the unique optimal solution of
the max-entropy optimization problem:
\[\begin{array}{rl}
\displaystyle\sup_{\A,\B\succeq0}&\{\,\log\mathrm{det}(\A)+\log\mathrm{det}(\B):\\
\mbox{s.t.}& 1\,=\,\frac{1}{2n+1}\,[\,\underbrace{\v_n(x)^T\A\,\v_n(x)}_{\sigma_0(x)}\\
&+
(1-x^2)\,\underbrace{\v_{n-1}(x)^T\B\,\v_{n-1}(x)}_{\sigma_1(x)}\,]\,,\quad\forall x\,\in\,\R\,\}\,\end{array}\,.\]
So the partition of unity \eqref{unity-1} associated with Chebyshev polynomials 
is associated with Putinar's certificate of positivity \eqref{eq:put} on $[-1,1]$, based on SOS polynomials, and applied to the constant polynomial $``1"$, whereas
the partition of unity \eqref{bern-2} associated with Bernstein polynomials 
is associated with Handelman's certificate of positivity on $[0,1]$, based on nonnegative coefficients $c_{ij}$
of $x^i (1-x)^j$.  But both share the same variational property, namely their coefficients
in their respective certificate maximize a similar entropy criterion.
\section{Notation definitions and a duality result}

\subsection{Notation and definitions}
Let $\R[x]$ be the ring of univariate polynomials, $\R[x]_t\subset\R[x]$ be the 
space of polynomials of degree at most $t$, and
$\Sigma[x]_t\subset\R[x]_{2t}$ be the convex cone 
of univariate sum-of-squares (SOS) polynomials of degree at most $2t$. An element $p\in\R[x]_t$
is written as $x\mapsto p(x)=\p^T\v_t(x)$ where $\v_t(x)=(x^j)_{0\leq j\leq t}$ is the usual monomial basis of $\R[x]_t$,
and $\p\in\R^{t+1}$ is the vector of coefficients of $p$ in that basis.
An element $\phi\in\R[x]^*_t$ is represented by
a vector $\bphi\,=\,(\phi_j)_{0\leq j\leq t}$, that is, $\phi(p)\,=\,\bphi^T\p$.

Given a polynomial $g\in\R[x]$
and a 
linear functional $\phi\in\R[x]_t^*$ with associated 
sequence $\bphi\in\R^{t+1}$, define the new linear functional $g\cdot\phi\in\R[\x]^*_t$
(with associated sequence $g\cdot\bphi$) defined by:
\[p\mapsto g\cdot\phi(p)\,=\,\phi(g\,p)\,=\,\langle g\cdot\bphi,\p\rangle\,,\quad \forall p\in\R[x]_t\,.\]

Given a sequence $\bphi\in\R[x]_t$ denote by $\M_t(\bphi)$ (or $\M_t(\phi)$), the \emph{moment} matrix
associated with $\bphi$. It  is the $(t+1)\times (t+1)$ real  symmetric Hankel matrix with  entries
\[\M_t(\bphi)[i,j]\,=\,\phi(x^{i+j-2})\,=\,\phi_{i+j-2}\,,\quad\forall 1\leq i,j\leq t+1\,.\]
Similarly, the matrix $\M_t(g\cdot\bphi)$ (or $\M_t(g\cdot\phi)$), i.e., the moment matrix associated with the sequence
$g\cdot\bphi$, is also called the \emph{localizing} matrix associated with $\bphi$ and $g$.

With $x\mapsto g(x):=(1-x^2)$, introduce the convex cone $Q_n(g)\subset\R[x]_{2n}$ 
defined by
\[Q_n(g)\,:=\,\{\,\sigma_0+\sigma_1\,g\::\: \sigma_0\in\Sigma[x]_n\,,\:\sigma_1\in\Sigma[x]_{n-1}\,\}\,.\]
Its dual $Q_n(g)^*\subset\R[x]_{2n}^*$ is the convex cone defined by
\[Q_n(g)^*\,:=\,\{\,\bphi\in\R^{2n+1}:\M_t(\bphi)\succeq0\,;\:
\M_{n-1}(g\cdot\bphi) \succeq0\,\}\,.\]
In the terminology of real algebraic geometry, $Q_n(g)$
is the \emph{quadratic module} associated with
the polynomial $g$. 

\subsection{Two certificates of positivity}
We next introduce two (celebrated) certificates of positivity on  $[1,1]$ and $[0,1]$ respectively.

\begin{theorem}[Markov-Luk\'acs \& Putinar]
\label{th1}
If $p\in\R[x]_{2n}$ is nonnegative on $[-1,1]$ then $p\in Q_n(g)$, i.e.,
\begin{equation}
 \label{th1-1}
 p\,=\,\sigma_0+\sigma_1\,g\,,
 \end{equation}
 for some SOS polynomials $\sigma_0\in\Sigma[x]_n$ and $\sigma_1\in\Sigma[x]_{n-1}$. In fact,
 we even have
\begin{equation}
 \label{th1-11}
 p\,=\,p_0^2+p_1^2\,g\,,
 \end{equation}
 for some polynomials $p_0\in\R[x]_n$ and $p_1\in\R[x]_{n-1}$.
\end{theorem}
So the refinement \eqref{th1-11} of \eqref{th1-1} is Markov-Luk\'acs' theorem
which states that 
one may even decompose $p$ as a (weighted) sum of only two \emph{single} squares. Putinar's Positivstellensatz \cite{putinar} is a multivariate generalization of Theorem \ref{th1} for 
polynomials that are strictly positive on a compact basic semi-algebraic set (whose generators satisfy an Archimedean property).

\begin{theorem}[Bernstein \cite{bernstein}]
\label{th2}
If $p\in\R[x]_{n}$ is (strictly) positive on $[0,1]$ then there exists $m\geq n$ and $0\leq \c=(c_{ij})_{i+j=m}$ such that
\begin{equation}
 \label{th2-1}
 p(x)\,=\,\sum_{i+j=m}c_{ij}\,x^i \,(1-x)^j\,,\quad\forall x\in\R\,.
\end{equation}
\end{theorem}

\begin{theorem}[Handelman (univariate)]
\label{th-handelman}
If $p\in\R[x]_{n}$ is (strictly) positive on $[0,1]$ then there exists $0\leq \c=(c_{ij})_{i+j\leq n}$ such that
\begin{equation}
 \label{th-handelman-1}
 p(x)\,=\,\sum_{i+j\leq n}c_{ij}\,x^i \,(1-x)^j\,,\quad\forall x\in\R\,.
\end{equation}
\end{theorem}
Theorem \ref{th-handelman} is a specialization to the univariate case and $S=[0,1]$, of the more general
Handeman's Positivstellensatz \cite{handelman} valid on a convex polytope $S\subset\R^d$ with nonempty interior, while in (the older) Theorem \ref{th2} of Bernstein, all terms $x^i\,(1-x)^j$ have same degree $i+j=m$. 
Note that the two certificates of positivity \eqref{th1-1} (or \eqref{th1-11}) and \eqref{th-handelman-1} are quite different in nature. 
The first one \eqref{th1-1} which is the univariate version of Putinar's theorem, 
uses SOS polynomials ($\sigma_0,\sigma_1$) and is valid for polynomials that 
are \emph{nonnegative} on $[-1,1]$,
whereas \eqref{th-handelman-1} which uses a vector $\c$ of nonnegative scalars, is valid for polynomials
that are strictly positive on $[0,1]$. 

Moreover, 
testing whether a given $p\in\R[x]_{2t}$ satisfies \eqref{th1-1},
 reduces to solving a \emph{semidefinite program} (or an eigenvalue problem).
On the other hand, testing whether $p$
satisfies \eqref{th2-1} reduces to solving a \emph{linear program}.

\section{Main result}

In this section we show how Chebyshev (resp. Bernstein) polynomials are related in a similar manner to
the equilibrium measure of $[-1,1]$ (resp. Lebesgue measure on $[0,1]$).

\subsection{Chebyshev polynomials and equilibrium measure of $[-1,1]$}

\paragraph{Polynomial Pell's equation}

Let $(T)_{j\in\N}$ (resp. $(U_j)_{j\in\N}$) be the Chebyshev polynomials 
of the first (resp. second) kind. After normlization
$\widehat{T}_j:=T_j/\sqrt{2}$, $j=1,\ldots,n$, and 
$\widehat{U}_j:=U_j/\sqrt{2}$, $j=0,\ldots,n$, 
$(\widehat{T}_j)_{j\in\N}$ (resp. $(\widehat{U}_j)_{j\in\N}$ ) form a family of polynomials orthonormal w.r.t. $d\phi=dx/\pi\sqrt{1-x^2}$ (resp. $(1-x^2)\,d\phi=\sqrt{1-x^2}dx/\pi$).
It turns out that the Chebyshev polynomials satisfy the so-called polynomial Pell's equation
\eqref{eq:pell}, that is,
\[T_n(x)^2\,+(1-x^2)\,U_{n-1}(x)^2\,=\,1\,,\quad \forall x\in \R\,,\quad\forall n\geq1\,.\]
As already mentioned in introduction, observe that \eqref{eq:pell} is a nice illustration of Markov-Luk\'acs's theorem \eqref{th1-11}
for the constant polynomial $``1"$ which is indeed positive on $[-1,1]$. In other words,
the Chebyshev polynomials of first and second kind provide the Markov-Luk\'acs decomposition 
of the constant polynomial $``1"$.

In \cite{pell-2} this result was given an interpretation in terms of Christoffel functions of
the equilibrium measure $\phi$ of $[-1,1]$, namely:
\begin{theorem}[\cite{pell-2}]
\label{th-equi}
Let $x\mapsto g(x):=1-x^2$. For every $n\in\N$:
\begin{eqnarray}
\label{cheby-1}
1&=&\frac{1}{2n+1}
[\,\sum_{j=0}^n \widehat{T}_j^2+g\,\sum_{j=0}^{n-1}\widehat{U}_j^2\,]\\
\label{cheby-2}
&=&\frac{1}{2n+1}\,[\,\v_n(x)^T\M_n(\phi)^{-1}\v_n(x)\\
&&+\,g(x)\,\v_{n-1}(x)^T\M_{n-1}(g\cdot\phi)^{-1}\v_{n-1}(x)\,]\,,\quad\forall x\in\R\\
\label{cheby-3}
&=&\Lambda^{\phi}_n(x)^{-1}+g(x)\,\Lambda^{g\cdot\phi}_{n-1}(x)^{-1}\,,\quad\forall x\in\R\,.
\end{eqnarray}
\end{theorem}
So Theorem \ref{th-equi} states that the constant polynomial $``1"$ has a distinguished 
certificate of positivity on $S=[-1,1]$. Among all of its possible Putinar's representations
\eqref{th1-1}, the one in \eqref{cheby-1}-\eqref{cheby-3}  is directly related to the equilibrium measure $d\phi=dx/\pi\sqrt{1-x^2}$ of the interval $[-1,1]$.
In addition, as we next show, this distinguished certificate satisfies an \emph{extremal property}. 
\begin{lemma}[\cite{cras-1}]
\label{lemma1}
The couple of Gram matrices $(\A^*,\B^*)$ with $\A^*:=\M_n(\phi)^{-1}$ and $\B^*:=\M_{n-1}((1-x^2)\cdot\phi)^{-1}$,
is the unique optimal solution of the convex optimization problem:
\[\begin{array}{rl}
\displaystyle\sup_{\A,\B\succ0}&\{\,\log\mathrm{det}(\A)+\log\mathrm{det}(\B):\\
\mbox{s.t.}&2n+1=\underbrace{\v_n(x)^T\A\,\v_n(x)}_{\sigma_0(x)}
+g(x)\,\underbrace{\v_{n-1}(x)^T\B\,\v_{n-1}(x)}_{\sigma_1(x)}\,\:
\forall x\in\R\,\}\end{array}.\]
\end{lemma}
The proof of Lemma \ref{lemma1} in \cite{cras-1} combines (1) a duality result of  Nesterov \cite{nesterov} which establishes 
a one-to-one correspondence between the interiors of the convex cone of polynomials $Q_n(g)$ and
its dual $Q_n(g)^*$, and (ii) the generalized Pell's equation \eqref{cheby-1}-\eqref{cheby-2} which allows to 
identify the element $\bphi=(\phi_j)_{j\leq 2n+1}$ in  $Q_n(g)^*$ associated with $1\in Q_n(g)$,
to be moments of the equilibium measure $d\phi=dx/\pi\sqrt{1-x^2}$ of the interval $[-1,1]$.

\subsection*{Partition of unity}

Observe that the polynomials $\{(T_i)_{i\leq n}/n+1,(g\,U_j/n+1)_{j\leq n-1}\}$  or
$\{(\widehat{T}_i)_{i\leq n}/2n+1,(g\,\widehat{U}_j/2n+1)_{j\leq n-1}\}$, 
form a partition of unity of the interval
$[-1,1]$. Lemma \ref{lemma1} establishes that it maximizes an entropy criterion among all possible
polynomial partitions of unity in the form $\sigma_0+g\,\sigma_1$ (a certificate of positivity 
on $S=[-1,1]$ for the polynomial $``1"$).

\subsection{Bernstein polynomials and Lebesgue measure on $[0,1]$}
Let $S=[0,1]$ and $s(n):={2+n\choose n}$.
The family of Bernstein polynomials $B_{n,j}\subset\R[x]$ is defined by:
\[x\mapsto B_{n,j}(x)\,:=\,{n\choose j}\,x^j\,(1-x)^{n-j}\,,\quad\forall j\leq n\,,\: n\,\in\,\N\,.\]
Among their numerous properties, they form a basis of $\R[x]_n$, they are nonnegative on $[0,1]$, bounded by $1$, and in addition:
\begin{equation}
\label{part-unity}
 1\,=\,\sum_{j=0}^n B_{n,j}(x)\,,\quad\forall x\in\R\,,\quad \forall n\in\N\,,
\end{equation}
so that they form a \emph{partition of unity} of the interval $[0,1]$. Moreover,
\begin{equation}
 \int_0^1 B_{n,j}(x)\,dx\,=\,\frac{1}{n+1}\,,\quad\forall j=0,\ldots,n\,,\quad \forall n\in\N\,.
\end{equation}
Interestingly, the \emph{envelope} $f_n$ of the Bernstein polynomials $B_{n,j}$ is the Chebyshev density
\[f_n(x)\,:=\,\frac{1}{n}\cdot \frac{1}{\sqrt{2\pi x(1-x)}}\,.\]

Next, for $n\in\N$ fixed and $S=[0,1]$, consider 
the convex cones $\mathscr{C}_n$ and its dual $\mathscr{C}^*_n$ defined by:
\begin{eqnarray*}
\mathscr{C}_n&=&\{\,\sum_{(i,j)\in\N^2_n}c_{ij}\,x^i\,(1-x)^{j}\::\: \c\geq0\,\}\:\subset\R[\x]_n\\
\mathscr{C}^*_n&=&\{\,\bphi\in\R[x]^*_n\::\: \phi(x^j\,(1-x)^{j})\,\geq\,0\,,\quad\forall (i,j)\in\N^2_n\ \}\:\subset\R[\x]_n^*\,.
\end{eqnarray*}
\begin{remark}
In view of Bernstein's Theorem \ref{th2},
we could also consider the smaller convex cone
\[\{\,\sum_{j=0}^nc_{j}\,x^j\,(1-x)^{n-j}\::\: \c\geq0\,\}\:\subset\R[\x]_n\,,\]
where all terms $x^j(1-x)^{n-j}$ have same degree $n$.
But since in Theorem \ref{th-handelman}, Handelman's Positivstellensatz
requires to consider all positive linear combinations of  powers $x^i (1-x)^j$ with $i+j\leq n$,
we will rather consider $\mathscr{C}_n$ as defined above.
\end{remark}
\begin{proposition}
\label{lem-aux}
$p\in\mathrm{int}(\mathscr{C}_n)$ if and only if there exists $0<\c=(c_{ij})_{i+j\leq n}$ such that
\[p\,=\,\sum_{(i,j)\in\N^2_n}c_{ij}\,x^i (1-x)^j\,,\quad\forall x\in\R\,.\]
\end{proposition}
\begin{proof}
{\it Only if part:} 
 If $p\in\mathrm{int}(\mathscr{C}_n)$ then 
  $p-\varepsilon\in \mathscr{C}_n$ for $\varepsilon>0$ sufficiently small, that is,
 \[p-\varepsilon\,=\,\sum_{(i,j)\in\N^2_n}c_{ij}\,x^i (1-x)^j\,,\quad\forall x\in\R\,.\]
 for some $\c\geq0$.  Next, using \eqref{part-unity} yields
\[n+1\,=\,\sum_{j=0}^n\sum_{k=0}^j{j\choose k}x^j(1-x)^{j-k}\,,\quad\forall x\in\R\,,\]
and therefore we obtain
\begin{eqnarray*}
p&=&\sum_{(i,j)\in\N^2_n}c_{ij}\,x^i (1-x)^j+\frac{\varepsilon}{n+1}\sum_{j=0}^n\sum_{k=0}^j{j\choose k}x^j(1-x)^{j-k}\,,\quad\forall x\in\R\\
&=&\sum_{(i,j)\in\N^2_n}(\underbrace{c_{ij}+\varepsilon_{ij}}_{\tilde{c}_{ij}>0})\,x^i (1-x)^j\,,\quad\forall x\in\R\,,
\end{eqnarray*}
for some $\tilde{\c}>0$.

\emph{If part:} Let $p=\sum_{(i,j)\in\N^2_n}c_{ij}\,x^i (1-x)^j$ with $\c>0$, and let $\varepsilon>0$ and $q\in\R[x]_n$ be such that $\Vert p-q\Vert<\varepsilon$. As $(B_{n,j})_{j\leq n}$ form a basis of $\R[x]_n$,
\[p-q\,=\,\sum_{i+j=n}\tau_{ij}\,x^i (1-x)^j\,,\]
for some $\btau\in\R^{n+1}$ with $\sup_{i+j=n}\vert\tau_{ij}\vert<\kappa\,\varepsilon$ (for some $\kappa>0$).
Defining $\tau_{ij}:=0$ whenever $(i,j)\in\N^2_n$ with $i+j<n$, one obtains
\[q\,=\,p-\sum_{i+j=n}\tau_{ij}\,x^i (1-x)^j\,=\,\sum_{(i,j)\in\N^2_n}(\underbrace{c_{ij}-\tau_{ij}}_{\tilde{c}_{ij}})\,x^i \,(1-x)^j\,,\]
where $\tilde{\c}=(\tilde{c}_{ij})\geq0$ provided that $\varepsilon>0$ is small enough. Therefore we conclude that $q\in\mathscr{C}_n$ whenever $\varepsilon$ is small enough, and so $p\in\mathrm{int}(\mathscr{C}_n)$.
\end{proof}
\begin{lemma}
\label{univariate-nesterov}
Let $p\in\mathrm{int}(\mathscr{C}_n)$ be fixed. Then  the optimization problem
\begin{equation}
\label{def-primal}
\begin{array}{rl}
\P:\quad \rho_n\,=\,\displaystyle\sup_{\c>0}&\{\,\displaystyle\sum_{(i+j)\in\N^2_{n}}\log{c_{ij}}:\:\\
\mbox{s.t.}& p(x)\,=\,
\displaystyle\sum_{(i,j)\in\N^2_n}c_{ij}\,x^i\,(1-x)^j\,,\quad\forall x\in\R\,\}\end{array}\,,
\end{equation}
is a convex optimization problem whose unique optimal solution $\c^*>0$ satisfies 
\begin{equation}
\label{univariate-nesterov-1}
 1/c^*_{ij}\,=\,\phi^*_p(x^i \,(1-x)^j)\,,\quad \forall (i,j)\in\N^2_n\,,
\end{equation}
for some element $\phi^*_p\in\mathscr{C}^*_n$, and therefore
\begin{equation}
\label{univariate-nesterov-2}
 p(x)\,=\,\sum_{(i,j)\in\N^2_n}\frac{x^i\,(1-x)^j}{\phi^*_p(x^i \,(1-x)^j)}\,,\quad \forall x\in\R\,.
\end{equation}
\end{lemma}
\begin{proof}
 We first prove the $\rho_n$ is finite. As 
  $p\in\mathrm{int}(\mathscr{C}_n)$, by Lemma \ref{lem-aux} there exists $\hat{\c}>0$ such that
 $p=\sum_{(i,j)\in\N^2_n}\hat{c}_{ij}\,x^i (1-x)^j$,  and so  Slater condition\footnote{Slater condition holds 
 for the convex optimization problem $\min\{\,f(\x): g_j(\x)\leq0\,,\:j\in J\}$ if there exists 
 $\x_0$ such that $g_j(\x_0)<0$, for all $j\in J$.} holds for $\P$ and
 $\rho_n\geq\sum_{(i,j)\in\N^2_n}\log{\hat{c}_{ij}}>-\infty$. 
 Therefore we may and will consider only the (nonempty) subset of feasible solutions 
 \[\Delta\,:=\,\{\,\c\geq0:\: \sum_{(i,j)\in\N^2_n}\log{c_{ij}}\,
 \geq\,\sum_{(i,j)\in\N^2_n}\log{\hat{c}_{ij}}\,\}\,.\]
   Moreover, for any such feasible solution $\c\in\Delta$ of $\P$, with $x_0\in (0,1)$ fixed, 
 \[p(x_0)\,=\,\sum_{(i,j)\in\N^2_n}c_{ij}\,x_0^i\,(1-x_0)^j\quad\Rightarrow\quad 
 c_{ij}\,<\,\frac{p(x_0)}{x_0^i (1-x_0)^j}\,,\quad\forall (i,j)\in\N^2_n\,,\]
  and therefore  the  set $\Delta$ is compact, which in turn implies that $\P$ has an optimal solution $\c^*\in\Delta$ (hence with $\c^*>0$).
  Next, the necessary Karush-Kuhn-Tucker (KKT)-optimality conditions 
 impose that there exists $\phi^*_p\in\mathscr{C}^*_n$ such that \eqref{univariate-nesterov-1} holds, which in turn yields \eqref{univariate-nesterov-2}. Then uniqueness of $\c^*$ follows from the fact that the objective function is strictly concave.
 \end{proof}

Lemma \ref{univariate-nesterov} is the analogue for the cones $\mathscr{C}_n$ and $\mathscr{C}_n^*$
of Nesterov's one-to-one correspondence between the cones $Q_n(g)$ and $Q_n(g)^*$.
Of course \eqref{univariate-nesterov-2} raises a natural question: \emph{What is the element $\phi^*_p\in\mathscr{C}_n^*$ associated with
$p\in\mathrm{int}(\mathscr{C}_n)$? } We answer this question for the constant polynomial $``1"$.
\begin{theorem}
\label{th1}
Let $s(n)=(n+1)(n+2)/2$ and let $p\in\R[x]_n$ be the constant polynomial $x\mapsto p(x)=s(n)$ for all $x\in\R$. Then for every $n\in\N$, the vector $\c^*\in\R^{s(n)}$ with  
\[1/c^*_{ij}\,=\,\phi^*_p(x^i (1-x)^j)\,=\,\int_0^1x^i\,(1-x)^j\,dx\,,\quad \forall (i,j)\in\N^2_n\,,\]
is the unique optimal solution 
of \eqref{def-primal}, and:
\begin{eqnarray}
\label{eq:th1-1}
 1&=&\frac{1}{s(n)}\sum_{(i,j)\in\N^2_n}\frac{x^i \,(1-x)^j}{\int_0^1 x^i(1-x)^jdx}\\
\nonumber
 &=&\frac{1}{s(n)}
 \sum_{(i,j)\in\N^2_n}\frac{B_{i+j,i}(x)}{\int_0^1B_{i+j,i}(x)dx}\,,\quad\forall x\in\R\,.
\end{eqnarray}
That is, $\phi^*_p$ is the Lebesgue measure on $[0,1]$, and the polynomials $(\frac{(k+1)}{s(n)}B_{k,j})$ form a partition  of unity of $[0,1]$.
\end{theorem}
\begin{proof}
From the proof of Lemma \ref{univariate-nesterov} we have seen that $\P$ in \eqref{def-primal} is a convex optimization problem with a unique optimal solution which satisfies the KKT-optimality conditions \eqref{univariate-nesterov-1}. Next, since 
\[n+1\,=\,\sum_{k=0}^n\sum_{j=0}^k\,B_{kj}(x)\,=\,\sum_{(i,j)\in\N^2_n}\hat{c}_{ij} \,x^i (1-x)^j\,,\]
with $\hat{\c}>0$, Slater condition holds for $\P$. This in turn implies that
the first-order KKT optimality condition are not only necessary but also sufficient. 
So let $\phi^*$ be the Lebesgue mesure on $[0,1]$. We next prove that
$\c^*=(c^*_{ij})$ with $1/c^*_{ij}:=\phi^*(x^i (1-x)^j)$, for all $(i,j)\in\N^2_n$ 
is feasible for $\P$ and hence is the unique optimal solution of $\P$. Indeed
\[1/c^*_{ij}\,:=\,\phi^*(x^i\,(1-x)^j)\,=\,\frac{\phi^*(B_{i+j,i})}{{i+j\choose i}}\,=\,\frac{1}{(i+j+1)\cdot{i+j\choose i}}\,,\]
and therefore:
\begin{eqnarray*}
\sum_{(i,j)\in\N^2_n} c^*_{ij}\,x^i\,(1-x)^{j}&=&\sum_{(i,j)\in\N^2_n} (i+j+1)\,B_{i+j,i}(x)\\
&=&\sum_{k=0}^n\sum_{(i,j)\in\N^2_n:i+j=k} (i+j+1)\,B_{i+j,i}(x)\\
&=&\sum_{k=0}^n(k+1)\sum_{j=0}^k B_{k,j}(x)\\
&=&\sum_{k=0}^n(k+1)\,=\,\frac{(n+1)(n+2)}{2}\,=\,s(n)\,.
\end{eqnarray*}
 \end{proof}
\begin{remark}
Theorem \ref{th1} reveals that the linear functional $\phi^*_p$ of Lemma \ref{univariate-nesterov} associated with
the constant polynomial $p=s(n)$, is the Lebesgue measure on $[0,1]$. So \eqref{eq:th1-1} is indeed the analogue for Bernstein polynomials and Lebesgue measure on $S=[0,1]$, of \eqref{cheby-1} 
for Chebyshev polynomials and the equilibrium measure $dx/\pi\sqrt{1-x^2}$ on $S=[-1,1]$. Both 
resulting partitions of unity maximize an entropy criterion of a very similar flavor.
\end{remark}

\section{Extension to the canonical simplex}

In \cite{cras-3} we have proved a similar (but only partial)  result for the $2$-dimensional canonical simplex
$S:=\{(x,y): x+y\leq 1\,;\: x,y\,\geq0\,\}$ whose equilibrium measure
is $d\phi(x,y)=dx\,dy/\pi\sqrt{x\,y\,(1-x-y)}$.

Namely 
let $s(n):={2+n\choose 2}$,  and introduce the quadratic polynomials
$(x,y)\mapsto g_1(x,y):=x\,y$, 
$(x,y)\mapsto g_2(x,y)\:=x\,(1-x-y)$, and
$(x,y)\mapsto g_3(x,y):=y\,(1-x-y)$.
For $n=1,2,3$, in \cite{cras-3} we have obtained
\begin{equation}
\label{simplex-equi}
s(n)+s(n-1)\,=\,\Lambda^{\phi}_n(x,y)^{-1}+\sum_{i=1}^3 g_i(x,y)\,\Lambda^{g_i\cdot\phi}_{n-1}(x,y)^{-1}\,,
\end{equation}
for all $(x,y)\in\R^2$, and indeed,
\eqref{simplex-equi} is a perfect analogue for the simplex, of \eqref{cheby-3} for the interval $[-1,1]$. 

We next prove the analogue of \eqref{eq:th1-1} for the $d$-dimensional simplex, and $n=1,2$.
We will use the following known (intermediate) result.
\begin{proposition}
\label{prop-simplex}
Let $\phi^*$ be the uniform probability measure on 
$S=\{\,\x\in\R^d_+:\sum_ix_i\leq 1\,\}$, 
with moments $\bphi^*=(\phi^*_{\balpha})_{\balpha\in\N^d}$. Then
\begin{equation}
\label{int-simplex}
\phi^*_{\balpha}\,=\,\phi(\x^{\balpha})\,=\,\frac{d\mathrm{!}\,\alpha_1\mathrm{!}\cdots\alpha_d\mathrm{!}}{(d+\vert\balpha\vert)\mathrm{!}}\,,\quad\forall \balpha\in\N^d\,.
\end{equation}
\end{proposition}

Next, for each $n\in\N$, let $\hat{s}(n):={d+1+n\choose n}$,
i.e., $\hat{s}(n)$ is the dimension of $\R[x_1,\ldots,x_{d+1}]_n$ as a vector space. Let 
$\x\mapsto g_j(\x):=x_j$, $j=1,\ldots,d$, and $\x\mapsto g_{d+1}(\x):=1-\sum_{j=1}^d x_j$, so that
$S=\{\,\x\in\R^d: g_j(\x)\geq0\,,\quad j=1,\ldots,d+1\,\}$. Next, for every $\balpha\in\N^{d+1}$,
define the polynomial $\g^{\balpha}\in\R[\x]$ by:
\[\x\mapsto \g(\x)^{\balpha}\,:=\,g_1(\x)^{\alpha_1}\cdot g_2(\x)^{\alpha_2}\cdots g_{d+1}(\x)^{\alpha_{d+1}}\,.\]
Similarly define the convex cone $\mathscr{C}_n\subset\R[\x]_n$ by:
\[\mathscr{C}_n\,:=\,\{\,\sum_{\balpha\in\N^{d+1}_n}c_{\balpha}\,\g(\x)^{\balpha}\,:\quad \c=(c_{\balpha})_{\balpha\in\N^{d+1}_n}\geq0\,\}\,.\]
Its dual $\mathscr{C}^*_n$ is the convex cone defined by:
\[\mathscr{C}^*_n\,:=\,\{\,\bphi\in\R^{s(n)}:\quad \phi(\g^{\balpha})\,\geq\,0\,,\quad 
\forall\balpha\in\N^{d+1}_n\,\}\,.\]

\begin{theorem}
\label{th-simplex}

Let $\phi^*$ be probability with uniform distribution on the simplex $S$. With $n=1,2$, the optimization problem:
\begin{equation}
\P:\quad \sup_{\c\geq0}\,\{\,\displaystyle\sum_{\balpha\in\N^{d+1}_n}\log{c_{\balpha}}:\quad 
\hat{s}(n)\,=\,\sum_{\balpha\in\N^{d+1}_n}c_{\balpha}\,\g(\x)^{\balpha}\,,\quad\forall \x\in\R^d\,\}\,,
\end{equation}
has a unique optimal solution $0<\c^*\in\R^{\hat{s}(n)}$ which satisfies $1/c^*_{\balpha}=\phi^*(\g^{\balpha})$ for all $\balpha\in\N^{d+1}_{\balpha}$, and
\begin{equation}
 \label{simplex-1-2}
 1\,=\,\frac{1}{\hat{s}(n)}\,\sum_{\balpha\in\N^{d+1}_n}\frac{\g(\x)^{\balpha}}{\phi^*(\g^{\balpha})}\,,\quad\forall \x\in\R^d\,.
\end{equation}
Therefore the $\hat{s}(n)$ polynomials $\frac{1}{\hat{s}(n)}\,\{\frac{\g^{\balpha}}{\phi^*(\g^{\balpha})}\}_{\balpha\in\N^{d+1}_n}$ provide the simplex $S$ with a partition of unity that maximizes an entropy criterion 
and is strongly related to the uniform distribution $\phi^*$ on $S$.
\end{theorem}
\begin{proof}
We will show that $\c^*$ and $\phi^*$ satisfy the KKT-optimality conditions associated with $\P$, and as Slater condition holds for the convex optimization problem $\P$, it implies that $\c^*$ is an optimal solution of $\P$. Uniqueness follows from the fact that the objective function $\c\mapsto \sum_{\balpha}\log{c_{\balpha}}$ is strictly concave.
 
If $\c^*>0$ is an optimal solution, the KKT-optimality conditions state that
\begin{eqnarray}
\nonumber
1&=&\frac{1}{\hat{s}(n)}\,\sum_{\balpha\in\N^{d+1}_n}\c^*_{\balpha}\:\g(\x)^{\balpha}\,,\quad\forall\x\in\R^d\\
\label{aux}
1/c^*_{\balpha}&=&\phi(\g^{\balpha})\,,\quad\forall\balpha\in\N^{d+1}_n\,,\end{eqnarray}
for some element $\bphi\in\R[\x]^*_n$ such that $\phi(\g^{\balpha})\geq0$ for all $\balpha\in\N^{d+1}_n$. Conversely
under Slater condition, if \eqref{aux} holds then $\c^*$ is an optimal solution of $\P$.
So let $\phi^*$ be the probability measure uniformly supported on the simplex $S$ (i.e. Lebesgue measure on $S$, scaled to a probability measure). 

$\bullet$ With $n=1$ and invoking Proposition \ref{prop-simplex}, one obtains
\begin{eqnarray*}
1+\frac{(1-\sum_{i=1}^dx_i)}{\phi^*(1-\sum_{i=1}^dx_i)}
+\sum_{j=1}^d \frac{x_i}{\phi^*(x_i)}&=&
1+\frac{(1-\sum_{i=1}^dx_i)}{1/(d+1)}+\sum_{j=1}^d \frac{x_i}{1/(d+1)}\\
&=&1+d+1\,=\,d+2\,=\,\hat{s}(1)\,,\end{eqnarray*}
which shows that $\P$ has a feasible solution  with $\c>0$ (i.e., Slater condition holds for $P$), 
and \eqref{aux} holds with $\phi=\phi^*$, the desired result.

$\bullet$ Similarly, with $n=2$,
\[1+\frac{(1-\sum_{i=1}^dx_i)}{\phi^*(1-\sum_{i=1}^dx_i)}
+\sum_{j=1}^d \frac{x_i}{\phi^*(x_i)}
+\frac{(1-\sum_{i=1}^dx_i)^2}{\phi^*((1-\sum_{i=1}^dx_i)^2)}
+\sum_{j=1}^d \frac{x_i^2}{\phi^*(x_i^2)}\]
\[+\sum_{i<j}\frac{x_i\,x_j}{\phi^*(x_ix_j)}
+\sum_{i}\frac{x_i\,(1-\sum_jx_j)}{\phi^*(x_i\,(1-\sum_jx_j))}\]
\[=d+2+\frac{(d+2)(d+1)}{2}(1-\sum_{i=1}^dx_i)^2
+\frac{(d+1)(d+2)}{2}\sum_{j=1}^d x_i^2\]
\[+(d+1)(d+2)\sum_{i<j}x_i\,x_j +(d+1)(d+2)\,\sum_{i}x_i\,(1-\sum_jx_j)\]
\[=\,d+2+\frac{(d+2)(d+1)}{2}\,=\,\frac{(d+3)(d+2)}{2}\,=\,\hat{s}(2)\,.\]
\end{proof}
So again, and exactly as for the interval $[0,1]$,  \eqref{simplex-1-2} provides the $d$-dimensional simplex 
$S$ with a polynomial partition of unity (of degree $n=1$ et $n=2$)
simply expressed in terms of the generators $\{\g^{\balpha}\}$ of the cone $\mathscr{C}_n$, 
scaled by $1/\phi^*(\g^{\balpha})$, where $\phi^*$ is the 
Lebesgue measure on $S$ (scaled to a probability measure).

\section{Conclusion}
We have shown that Chebyshev polynomials and Bernstein polynomials 
are strongly related to respectively the equilibrium measure of $S=[-1,1]$
and the Lebesgue measure on $S=[0,1]$. Both  
provide a specific partition of unity interpreted in terms of Putinar's certificate of positivity
for the former and Handelman's certificate of positivity for the latter, applied to the constant polynomial $``1"$.
In both cases the resulting specific partition of unity maximizes an entropy criterion
over all possible certificates of positivity for $``1"$. We have partially extended this result (and comparison) to the $d$-dimensional canonical  simplex for degrees $n=1,2$, and extension to higher degrees remains to be proved. Finally, extension to arbitrary convex polytopes in also a topic of further investigation.

\end{document}